\documentclass[a4paper,12pt]{amsart}
\usepackage{amssymb}
\usepackage{hyperref}
\usepackage[enableskew]{youngtab}
\usepackage{array}
\newcolumntype{C}[1]{>{\centering\arraybackslash$}p{#1}<{$}}
\newcolumntype{L}[1]{>{\arraybackslash$}p{#1}<{$}}
\usepackage[line,matrix,frame,curve,poly]{xy}
\usepackage{tikz}
\usepackage{lineno}

\usepackage[T1]{fontenc}

\newcommand{\fS}{{\mathfrak S}}

\newcommand{\GL}{\mathrm{GL}}
\newcommand{\Sp}{\mathrm{Sp}}
\newcommand{\ep}{\epsilon}

\newcommand{\ltd}{\lessdot} 
\newcommand{\gtd}{\gtrdot}
\newcommand{\K}{K} 
 
\newcommand{\fQ}{F} 

\newtheorem{theorem}{Theorem}[section]
\newtheorem{corollary}[theorem]{Corollary}

\newtheorem{lemma}[theorem]{Lemma}
\newtheorem{proposition}[theorem]{Proposition}
\theoremstyle{definition} 
\newtheorem{definition}[theorem]{Definition}
\newtheorem{example}[theorem]{Example}

\newcommand{\Dfn}[1]{\emph{#1}}
\def\totop#1{%
  \vtop{\vskip0pt\hbox{$#1$}}%
}

\DeclareMathOperator{\Des}{Des}
\DeclareMathOperator{\ch}{ch}
\DeclareMathOperator{\sh}{sh}
\DeclareMathOperator{\Sun}{Sun}
\DeclareMathOperator{\RS}{RS}

\DeclareMathOperator{\Ber}{Ber}

\DeclareMathOperator{\Rob}{Rob}
\DeclareMathOperator{\rev}{rev}
\DeclareMathOperator{\Osc}{Osc}
\DeclareMathOperator{\SYT}{SYT}
\DeclareMathOperator{\SSYT}{SSYT}

\DeclareMathOperator{\wt}{wt}
\DeclareMathOperator{\Tr}{Tr}

\newcommand{\wgtCn}{\mu}

\newcommand{\wgtSr}{\lambda}

\newcommand{\cT}{\mathcal{T}}

\newcommand{\dotcup}{\ensuremath{\mathaccent\cdot\cup}}

\title{Descent sets for symplectic groups} \author{Martin Rubey}
\address{Fakult\"at f\"ur Mathematik und Geoinformation, TU Wien,
  Austria}%
\email{martin.rubey@math.uni-hannover.de}%
\author{Bruce E. Sagan}%
\address{Department of Mathematics, Michigan State University, East
  Lansing, MI 48824-1027, USA}%
\email{sagan@math.msu.edu}%
\author{Bruce W. Westbury}%
\address{Department of Mathematics, University of Warwick, Coventry,
  CV4 7AL}%
\email{Bruce.Westbury@warwick.ac.uk}%
\date{\today}%

\begin{document}

\begin{abstract}
  The descent set of an oscillating (or up-down) tableau is
  introduced.  This descent set plays the same role in the
  representation theory of the symplectic groups as the descent set
  of a standard tableau plays in the representation theory of the general
  linear groups.  In particular, we show that the descent set is
  preserved by Sundaram's correspondence.
 This gives a direct combinatorial interpretation of the branching
  rules for tensor products of the defining representation of the symplectic groups;
  equivalently, for the Frobenius character of the action of a
  symmetric group on an isotypic subspace in a tensor power of the
  defining representation of a symplectic group.  
\end{abstract}
\maketitle

\section{Introduction} 
In this paper we propose a new combinatorial approach to an open
problem in representation theory and solve one particular case.  The
problem can be formulated in very general terms: let $G$ be a
connected reductive complex algebraic group and let $V$ be a finite
dimensional rational representation of $G$.
Then, for $r\ge 0$, the group $G$ acts diagonally on the tensor power
$\otimes^rV$; moreover the symmetric group, $\fS_r$, acts by
permuting tensor coordinates.  Since these two actions commute we
have an action of $G\times \fS_r$ on $\otimes^rV$.  This
representation is completely reducible.  The problem is then to
describe the decomposition of $\otimes^rV$ into irreducible
components.  In particular one would like to determine the characters
of the isotypic components as representations of $\fS_r$.

There are two examples we consider: the defining representations of
$G=\GL(n)$, a general linear group, and $G=\Sp(2n)$, a symplectic
group.  The results for the general linear groups are well known,
although not from the point of view taken here, while the results for
the symplectic groups are new.

Our approach to this problem builds on the combinatorial theory of
crystal graphs.  We introduce the notion of \Dfn{descent set} of a
highest weight vertex in a tensor power of a crystal graph.  When $V$
is the defining representation of a general linear group, the usual
descent set of a word is a descent set in this sense.  We then
exhibit a descent set for tensor products of the defining representation of a symplectic
group.  In particular, our main result, Theorem~\ref{thm:main-crystal} below, is obtained by showing that the Sundaram correspondence for oscillating tableaux preserves descent sets.

The rest of this article is structured as follows.  In the next section we will review our setup in the familiar case of the general linear group.  Section~\ref{sec:general} then considers any complex reductive algebraic group $G$ and contains our new definition of a descent function in this general setting.  Section~\ref{sec:oscillating} specializes again to the symplectic case and Section~\ref{sec:correspondences} gives the necessary background about the Sundaram and Berele correspondences.   The following section gives the proof of our main result by showing that  Sundaram's bijection preserves descent sets.    In the last section, we explore Roby's description of the Sundaram map using growth diagrams.  


\section{Robinson-Schensted and the general linear groups}

   
In this section we review some properties of the Robinson-Schensted
correspondence and its connection with Schur-Weyl duality.  For a
textbook treatment we refer to the books of Sagan~\cite[Chapter~3]{SymmetricGroup} or Stanley~\cite[Chapter~7]{EC2}.
Let $V$ be the defining representation of $\GL(n)$.  Then the result
due to Schur is that the decomposition of $\otimes^rV$ as a
representation of $\GL(n)\times \fS_r$ is given by
\begin{equation}
  \label{eq:Schur-Weyl}
  \otimes^rV \cong \bigoplus_{\substack{\mu\in
    P(r)\\\ell(\mu)\leq n}} V(\mu)\otimes S(\mu),  
\end{equation}
where $V(\mu)$ is an irreducible representation of $\GL(n)$ and
$S(\mu)$ is an irreducible representation of $\fS_r$. In the sum, $P(r)$ represents
the set of partitions  of $r$, and $\ell(\mu)$ is the number of parts of the partition $\mu$.

The Robinson-Schensted correspondence is a bijection
\begin{equation}
  \label{eq:RS}
  \RS: \{1,\dots,n\}^r \to \coprod_{\substack{\mu\in
      P(r)\\\ell(\mu)\leq n}} \SSYT(\mu,n)\times\SYT(\mu),
\end{equation}
where $\SSYT(\mu, n)$ is the set of semistandard Young tableaux of
shape $\mu$ with entries less than or equal to $n$, and $\SYT(\mu)$ is the set
of standard Young tableaux of shape $\mu$.  When $\RS(w) = (P, Q)$,
we call $P$ the insertion and $Q$ the recording tableau of $w$.

One relationship between equations~\eqref{eq:Schur-Weyl}
and~\eqref{eq:RS} is that the character of $V(\mu)$ is given by the
Schur polynomial associated to $\mu$,
\begin{equation}\label{eq:Schur-character}
  s_\mu(x_1,\dots,x_n) = \sum_{T\in\SSYT(\mu,n)} {\mathbf x}^T.
\end{equation}
so, in particular, $V(\mu)$ has a basis indexed by $\SSYT(\mu,n)$.
Also, there is a basis of $S(\mu)$ indexed by $\SYT(\mu)$.
Furthermore, the Robinson-Schensted correspondence is weight
preserving, so its existence constitutes a combinatorial proof of 
the character identity implied by equation~\eqref{eq:Schur-Weyl}.

Associated to any representation $\rho:\fS_r\to\GL(U)$ is its Frobenius
character 
\[
\ch\rho = \frac{1}{r!}\sum_{\pi\in\fS_r}\Tr\rho(\pi)
p_{\lambda(\pi)},
\]
where $\Tr$ denotes the trace, the partition $\lambda(\pi)$ is the
cycle type of the permutation $\pi$ and $p_{\lambda(\pi)}$ denotes
the associated power sum symmetric function.  By abuse of notation we
will write $\ch U$ instead of $\ch\rho$.  In particular, we have $\ch
S(\mu)=s_\mu$,
where $s_\mu$ is the Schur function associated to $\mu$.

To establish the connection with descent sets, we express the
Frobenius character of the isotypic component of type $\mu$ in equation~\eqref{eq:Schur-Weyl}
in terms of the
fundamental quasi-symmetric functions (see~\cite[Theorem~7.19.7]{EC2}):
\begin{equation}
  \label{eq:quasisymmetric}
  \ch S(\mu) =  s_{\mu} = \sum_{Q\in\SYT(\mu)} \fQ_{\Des(Q)}
\end{equation}
where, for any subset $D\subseteq\{1,2,\dots,r-1\}$, 
$$
F_D=\sum_{i_1\le i_2\le \dots \le i_r\atop j\in D \implies i_j<i_{j+1}\rule{0pt}{8pt}} x_{i_1}x_{i_2}\dots x_{i_r}.
$$

One of the starting points of this project was the remarkable fact,
due to Sch\"utzenberger~\cite[Remarque~2]{Schuetzenberger1963}, that
the usual descent set of a word equals the descent set of its recording tableau under the 
Robinson-Schensted correspondence.
Moreover, the correspondence restricts to a bijection between reverse
lattice permutations (also known as reverse Yamanouchi words) of
weight $\mu$ and the set $\SYT(\mu)$.

Thus, we can alternatively describe the Frobenius character as
\begin{equation}
  \label{eq:quasisymmetric-words}
  \ch S(\mu) =  \sum_{w} \fQ_{\Des(w)},
\end{equation}
where the summation is over all reverse lattice permutations of
weight $\mu$.  We provide symplectic analogues of equations~(\ref{eq:quasisymmetric}) and~(\ref{eq:quasisymmetric-words})
in Theorems~\ref{thm:main} and~\ref{thm:main-crystal}, respectively.


\section{The general setting}
\label{sec:general}


Let $V$ be a finite dimensional rational representation of a complex reductive
algebraic group $G$.  Let $\Lambda$ be the set of isomorphism classes
of irreducible rational representations of $G$ and let $V(\wgtCn)$ be
the representation corresponding to $\wgtCn\in\Lambda$.  For example,
when $G$ is the general linear group $\GL(n)$ we can identify
$\Lambda$ with the set of weakly decreasing sequences of integers of
length $n$.  Also, when $G$ is the symplectic group
$\Sp(2n)$ we can identify $\Lambda$ with the set of partitions with
at most $n$ parts.  In both cases, the trivial representation
corresponds to the empty partition and the defining representation
corresponds to the partition~$1$.

For the symmetric group, $\fS_r$, we
identify the isomorphism classes of its irreducible representations
with the set of partitions of $r$, $P(r)$.  We denote the representation corresponding to $\wgtSr\in
P(r)$ by $S(\wgtSr)$.

For each $r\ge 0$, the tensor power $\otimes^rV$ is a rational
representation of $G\times \fS_r$, where $G$ acts diagonally
and $\fS_r$ acts by permuting
tensor coordinates.  This representation is
completely reducible.  Decomposing it as a representation of $G$ we
obtain the following analogue of equation~\eqref{eq:Schur-Weyl}
\begin{equation}
  \label{eq:Berele}
  \otimes^rV \cong \bigoplus_{\wgtCn \in  \Lambda}
  V(\wgtCn)\otimes U(r,\wgtCn).
\end{equation}
The isotypic space $U(r,\wgtCn)$ inherits the action of $\fS_r$ and
therefore decomposes as
\begin{equation}
  \label{eq:Sundaram}
  U(r,\wgtCn) \cong \bigoplus_{\wgtSr\in P(r)} A(\wgtSr,\wgtCn)\otimes S(\wgtSr).
\end{equation}
Thus, the Frobenius character of $U(r,\wgtCn)$ is
\begin{equation}
  \label{eq:Frobenius}
  \ch U(r,\wgtCn) = \sum_{\wgtSr\in P(r)} a(\wgtSr,\wgtCn) s_\wgtSr,
\end{equation}
where $a(\wgtSr,\wgtCn)=\dim A(\wgtSr,\wgtCn)$.

When $V$ is the defining representation of $\GL(n)$ the coefficient
$a(\wgtSr,\wgtCn)$ equals $1$ for $\wgtSr=\wgtCn$ and vanishes
otherwise.  Thus, in this case the Frobenius character is simply
$s_\wgtCn$.  For the defining representation of the symplectic group
$\Sp(2n)$ the coefficients $a(\wgtSr, \wgtCn)$ were determined by
Sundaram~\cite{MR2941115} and Tokuyama~\cite{MR933441}.  In general
it is a difficult problem to determine these characters explicitly.

We would like to advertise a new approach to this problem, using
descent sets.  It appears that the proper setting for this approach
is the combinatorial theory of crystal graphs.  This theory,
introduced by Kashiwara, is an off-shoot of the representation theory
of Drinfeld-Jimbo quantised enveloping algebras.  For a textbook
treatment we refer to the book by Hong and Kang~\cite{MR1881971}.

For each rational representation $V$ of a connected reductive
algebraic group there is a crystal graph.  The vector space $V$ is
replaced by a set of cardinality $\dim(V)$.  The raising and lowering
operators, which are certain linear operators on $V$, are replaced by
partial functions on the set.  It is common practice to represent
these partial functions by directed graphs.  Each arc is coloured
according to the application of the Kashiwara operator it represents.

For example, the crystal graph corresponding to the defining
representation of $\GL(n)$ is
\begin{equation}
  \label{eq:crystal-GL}
  1\stackrel{1}{\longrightarrow}2\stackrel{2}{\longrightarrow} \dots
  \stackrel{n-2}{\longrightarrow}n-1\stackrel{n-1}{\longrightarrow} n.
\end{equation}
while the crystal graph corresponding to the defining representation
of $\Sp(2n)$ is
\begin{equation}
  \label{eq:crystal-Sp}
  1\stackrel{1}{\longrightarrow}2\stackrel{2}{\longrightarrow} \dots
  \stackrel{n-1}{\longrightarrow} n
  \stackrel{n}{\longrightarrow}
  {-n}\stackrel{n-1}{\longrightarrow}{-(n-1)}\stackrel{n-2}{\longrightarrow}
  \dots
\stackrel{1}{\longrightarrow}{-1}.
\end{equation}

Each vertex of the crystal has a weight and the sum of these weights
is the character of the representation, e.g.,
equation~\eqref{eq:Schur-character} for the general linear group.
Isomorphic representations correspond to crystal graphs that are
isomorphic as coloured digraphs and the representation is
indecomposable if and only if the graph is connected.

A vertex in a crystal graph with no in-coming arcs is a highest
weight vertex.  Each connected component contains a unique highest
weight vertex.  The weight of this vertex is the weight of the
representation it corresponds to.  Thus, the components of a crystal
graph are in bijection with the set of highest weight vertices.  

There is a (relatively) simple way to construct the crystal graph of
a tensor product of two representations given their individual
crystal graphs.  In particular, the highest weight vertices of the
crystal corresponding to $\otimes^rV$ can be regarded as words of
length $r$ with letters being vertices of the crystal corresponding
to $V$.

For example, when $V$ is the defining representation of $\GL(n)$, the highest
weight words can be identified with reverse Yamanouchi words, see
Definition~\ref{def:symplectic-LR}.  When $V$ is the  defining representation
of $\Sp(2n)$ the highest weight words correspond to $n$-symplectic
oscillating tableaux, see Definition~\ref{def:oscillating}.

Thus we obtain a combinatorial interpretation of
equation~\eqref{eq:Berele} which is a far-reaching generalisation of
the Robinson-Schensted correspondence.  Explicit insertion schemes in
analogy to the classical correspondence were found corresponding to  the defining
representation for the symplectic groups $\Sp(2n)$ as well as for the
odd and even orthogonal groups, see~\cite{MR2269126}.

An important feature of these insertion schemes is that they can be
understood as isomorphisms of crystals.  In the example of the
defining representation of $\GL(n)$, the Robinson-Schensted
correspondence puts the highest weight words in bijection with
standard tableaux.  Moreover, the words in each component of the
crystal have the same recording tableau, and two words have the same
insertion tableau if and only if they occur in the same position of
two isomorphic components of the crystal graph.

Let us remark that there are other insertion schemes for the
classical groups, for example by Berele~\cite{MR867655} for the
symplectic groups, Okada~\cite{MR1132576} for the even orthogonal
groups and Sundaram~\cite{MR1041447} for the odd orthogonal groups.
However, these are not isomorphisms of crystals.

We can now state the fundamental property we require for a descent
set in the general sense.
\begin{definition}\label{def:descent-set}
  Suppose that the function $\Des$  assigns to each highest weight vertex of
  the crystal graph corresponding to $\otimes^rV$ a subset of
  $\{1,2,\dots,r-1\}$.  Then $\Des$ is a \Dfn{descent function} if it satisfies
  \[
  \ch U(r,\wgtCn) = \sum_{w} \fQ_{\Des(w)},
  \]
  where the sum is over all highest weight vertices of weight
  $\wgtCn$.
\end{definition}
Thus, by equation~\eqref{eq:quasisymmetric-words}, the usual descent
set of a word is a descent set in the sense of this definition for
the defining representation of $\GL(n)$.  We note that in terms of
the crystal graph~\eqref{eq:crystal-GL} above, a highest weight
vertex $w_1 w_2\dots w_r$ has a descent at position $k$ if and only
if there is a (nontrivial) directed path from $w_{k+1}$ to $w_k$ in
the crystal graph.  

In this article we exhibit a descent function with a similar
description for the symplectic groups, see
Definition~\ref{def:desoscword}.


\section{Oscillating tableaux and descents}
\label{sec:oscillating}


In the case of the defining representation of the symplectic group
$\Sp(2n)$ the vertices of the crystal graph corresponding to
$\otimes^rV$ are words $w=w_1w_2\dots w_r$ in $\{\pm 1,\dots,\pm
n\}^r$.  The weight of a vertex is the tuple
$\wt(w)=(\wgtCn_1,\dots,\wgtCn_n)$, where $\wgtCn_i$ is the number of
letters $i$ minus the number of letters $-i$ in $w$.  The vertex is a
highest weight vertex if for any $k\leq r$, the weight of
$w_1\dots w_k$ is a partition, i.e., $\wgtCn_1\geq\wgtCn_2
\dots\geq\wgtCn_n\geq0$.

\begin{definition}\label{def:desoscword}
  A highest weight vertex $w_1w_2\dots w_r$ in the crystal graph
  corresponding to $\otimes^rV$ has a \Dfn{descent} at position $k$
  if there is a (nontrivial) directed path from $w_k$ to $w_{k+1}$ in
  the crystal graph~\eqref{eq:crystal-Sp}.

  The \Dfn{descent set} of $w$ is
  \[
  \Des(w)=\{k\ | \text{ $k$ is a descent of $w$}\}.
  \]
Note that the direction of the path for these descents is opposite from the one used for $\GL(n)$.
\end{definition}

We can now state our main result.
\begin{theorem}\label{thm:main-crystal}
  Let $V$ be the defining representation of the symplectic group
  $\Sp(2n)$.  Then the Frobenius character of the isotypic component, $U(r,\wgtCn)$,
 of $\otimes^r V$ is
  \[
  \ch U(r,\wgtCn) = \sum_w \fQ_{\Des(w)},
  \]
  where the sum is over all highest weight vertices of weight $\mu$
  in the corresponding crystal graph.
\end{theorem}
To prove this theorem we will first rephrase it in terms of
$n$-symplectic oscillating tableaux, also known as up-down-tableaux,
which are in bijection with the highest weight vertices in the
crystal graph corresponding to $\otimes^r V$.
\begin{definition}\label{def:oscillating}
  An \Dfn{oscillating tableau} of length $r$ and (final) \Dfn{shape}
  $\wgtCn$ is a sequence of partitions
  \[
  (\emptyset\!=\!\wgtCn^0,\wgtCn^1,\ldots,\wgtCn^r\!=\!\wgtCn)
  \] 
  such that the Ferrers diagrams of two consecutive partitions differ
  by exactly one box.   We view all Ferrers diagrams in English notation.

  The $k$-th step, going from $\wgtCn^{k-1}$ to $\wgtCn^k$, is an
  \Dfn{expansion} if a box is added and a \Dfn{contraction} if a
  box is deleted.  We will refer to the box that is added or
  deleted in the $k$-th step as $b_k$.

  The oscillating tableau $\cT=(\wgtCn^0,\wgtCn^1,\ldots ,\wgtCn^r)$
  is \emph{$n$-symplectic} if every partition $\wgtCn^i$ has at most
  $n$ non-zero parts.
\end{definition}

The next result follows immediately from the definitions above.

\begin{proposition}\label{prop:bij-oscillating-tableau-highest-weight-word}
  The following is a bijection between $n$-symplectic oscillating
  tableaux of length $r$ and highest weight vertices $w$ in the
  crystal graph corresponding to $\otimes^r V$: 
  \begin{itemize}
  \item given a highest weight vertex $w_1 w_2\dots w_r$, the
    oscillating tableau is the sequence of weights of its initial
    factors 
    \[
    (\wt(w_1),\, \wt(w_1w_2),\, \wt(w_1w_2w_3),\dots, \wt(w_1 w_2\dots w_r));
    \]
  \item for an oscillating tableau $\cT$ the corresponding word
    $w_1w_2\dots w_r$ is given by $w_k=\pm i$, where $i$ is the row
    of $b_k$ and one uses plus or minus if $b_k$ is added or deleted,
    respectively.
  \end{itemize}
\end{proposition}
Note that in general, the oscillating tableau obtained via Berele's
correspondence  (see Section~\ref{sec:correspondences}) from a word $w_1w_2\dots w_r$ is different from the
oscillating tableau given by the bijection above.

\begin{example}\label{ex:oscillating-tableaux}
  The $1$-symplectic oscillating tableaux of length three are
  \[
  (\emptyset,1,2,3),\quad (\emptyset,1,2,1),\quad\text{and}\quad
  (\emptyset,1,\emptyset,1).
  \]
In terms of Ferrers diagrams, these are
\[
\left(\emptyset\ ,\ \yng(1)\ ,\ \yng(2)\ ,\  \yng(3)\ \right),
\quad
\left(\emptyset\ ,\ \yng(1)\ ,\ \yng(2)\ ,\  \yng(1)\ \right),
\]
and
\[
\left(\emptyset\ ,\ \yng(1)\ ,\ \emptyset,\  \yng(1)\ \right),
\]
respectively.
  The corresponding words are
  \[
  1\,1\,1,\quad 1\,1\,\text{-}{1},\quad\text{and}\quad 1\,\text{-}{1}\,1.
  \]
 
 As a running example, we will use the oscillating tableau
  \[
 \cT= (\emptyset,1,11,21,2,1,2,21,211,21)
  \]
  which has length $9$ and shape $21$.  It is $3$-symplectic (since
  no partition has four parts) but it is not $2$-symplectic (since
  there is a partition with three parts).  The corresponding word is
  \[
  w=1\,2\,1\,\text{-}{2}\,\text{-}{1}\,1\,2\,3\,\text{-}{3}
  \]
  with descent set
  \begin{equation}
  \label{eq:Desw}
  \Des(w)=\{1,3,4,6,7,8\}.
  \end{equation}
\end{example}

We will now define the descent set of an oscillating tableau in such a way that a tableau $\cT$ and its word $w$ will always have the same descent set.

\begin{definition}\label{def:desosctab}
  An oscillating tableau $\cT$ has a \emph{descent} at position $k$ if
  \begin{itemize}
  \item step $k$ is an expansion and step $k+1$ is a contraction, or
  \item steps $k$ and $k+1$ are both expansions and $b_k$ is in a row
    strictly above $b_{k+1}$, or
  \item steps $k$ and $k+1$ are both contractions and $b_k$ is in a
    row strictly below $b_{k+1}$.
  \end{itemize}
   The
  \Dfn{descent set} of $\cT$ is
  \[
  \Des(\cT)=\{k\ | \text{ $k$ is a descent of $\cT$}\}.
  \]
\end{definition}
\begin{example}
  The descent set of the oscillating tableau $\cT$ from
  Example~\ref{ex:oscillating-tableaux} is 
  \[
  \Des(\cT)= \{1,3,4,6,7,8\}.
  \]
which is the same as the descent set in~\eqref{eq:Desw} as predicted.
\end{example}

The fact that these two descent sets always coincide is easy to prove directly from the definitions, so we omit the proof and just formally state the result.
\begin{proposition}\label{prop:des-w-des-t}
The descent set of an $n$-symplectic oscillating tableau of length $r$ coincides
  with the descent set of the corresponding highest weight vertex of
  the crystal graph of $\otimes^rV$.
\end{proposition}

Thus it suffices to prove the following variant of
Theorem~\ref{thm:main-crystal}.
\begin{theorem}\label{thm:main}
  Let $V$ be the defining representation of the symplectic group
  $\Sp(2n)$.  Then the Frobenius character of the isotypic component, $U(r,\wgtCn)$,
 of $\otimes^r V$ is
  \[
  \ch U(r,\wgtCn) = \sum_\cT \fQ_{\Des(\cT)},
  \]
  where the sum is over all $n$-symplectic oscillating tableaux of
  length $r$ and shape $\wgtCn$.
\end{theorem}

Let us motivate Definition~\ref{def:desosctab} in a second way.  Note
that a standard Young tableau $T$ can be regarded as an oscillating
tableau $\cT$ where every step is an expansion and the box containing
$k$ in $T$ is added during the $k$-th step in $\cT$.  In this case
\[
\Des(\cT)=\Des(T),
\]
where a descent of a standard Young tableau $T$ is an integer $k$
such that $k+1$ appears in a lower row than $k$ of $T$.  So the
definition of descents for oscillating tableaux is an extension of
the usual one for standard tableaux.

In Sundaram's correspondence, an arbitrary oscillating tableau $\cT$
is first transformed into a fixed-point-free involution $\iota$ on a
subset $A$ of the positive integers and a partial Young tableau $T$,
that is, a filling of a Ferrers shape with all entries distinct and
increasing in rows and columns.  There are natural notions of
descents for these objects which extend those for permutations in
$\fS_r$ and standard Young tableaux.
\begin{definition}\label{sec:des-partial}
 Let $A$ be a set of positive integers.  The \Dfn{descent set} of a bijection $\iota:A\rightarrow A$ is
  \[
  \Des(\iota)=\{k: k,k+1\in A,\; \iota(k)>\iota(k+1)\}.
  \]
  For a partial Young tableau $T$ whose set of entries is $A$, the
  \Dfn{descent set} is
  \[
  \Des(T)=\{k: k,k+1\in A,\; \text{$k+1$ is in a row below $k$}\}.
  \]
\end{definition}

The definition of the descent set for an oscillating tableau is constructed so
that it contains the union of
the descent sets of the associated partial tableau and involution under Sundaram's bijection.


\section{The correspondences of Berele and Sundaram}
\label{sec:correspondences}


One of our main tools for proving Theorem~\ref{thm:main} will be a
bijection, $\Sun$, due to Sundaram~\cite{MR2941115,MR1035496}.  
 Berele~\cite{MR867655} constructed a bijection which is a combinatorial counterpart of the isomrphism in equation~\eqref{eq:Berele} when $V$ is the defining
representation of $\Sp(2n)$
In combination with Berele's  correspondence, Sundaram's bijection can be regarded  in its turn as
a combinatorial counterpart of the isomorphism in
equation~\eqref{eq:Sundaram}.  In this section we define the objects
involved; the bijection $\Sun$  itself will be described in detail in the
next section.

\begin{definition}\label{def:symplectic-LR}
  Let $u$ be a word with letters in the positive integers.  Then $u$ is a
  \Dfn{Yamanouchi word} (or \Dfn{lattice permutation}) if, in each
  initial factor $u_1u_2\dots u_k$, there are at
  least as many occurrences of $i$  as there are  of
  $i+1$ for all $i\ge1$.
  The \Dfn{weight} $\beta$ of a lattice permutation $u$ is the
  partition $\beta=(\beta_1\ge\beta_2\ge\cdots)$, where $\beta_i$ is
  the number of occurrences of the letter $i$ in $u$.

 For a skew semistandard Young tableau $S$ the \Dfn{reading word}, $w(S)$, is  the word
obtained by concatenating the rows from bottom to top.  It has the nice property that applying the
Robinson-Schensted map to $w(S)$ one recovers $S$ as the insertion tableau.  We will need the
\Dfn{reverse reading  word}, $w^{rev}(S)$, which is obtained by reading $w(S)$ backwards.

A skew semistandard Young tableau $S$ of shape $\wgtSr/\wgtCn$ is
called  an \Dfn{$n$-symplectic Littlewood-Richardson tableau of weight
    $\beta$} if
  \begin{itemize}
  \item its reverse reading word is a lattice permutation of weight
    $\beta$, where $\beta$ is a partition with all columns having
    even length, and
  \item entries in row $n+i+1$ of $S$ are greater than or equal
    to $2i+2$ for $i\geq 0$. 
  \end{itemize}

 We will denote  the number of $n$-symplectic Littlewood-Richardson tableaux of
  shape $\wgtSr/\wgtCn$ and weight~$\beta$ by
  $c_{\wgtCn,\beta}^\wgtSr(n)$.
\end{definition}

For $\ell(\wgtSr)\leq n+1$ we have that the number $c_{\wgtCn,\beta}^\wgtSr(n)$ is
the usual Littlewood-Richardson coefficient whenever $\beta$ is a partition
with all columns having even length.  This is trivial for
$\ell(\wgtSr)\leq n$, and follows from the correspondence $\Sun$
described below for $\ell(\wgtSr)=n+1$.

Note that we are only interested in the case when the length of $\mu$
is at most $n$.  In this case the restriction on the size of the
entries is equivalent to the condition given by Sundaram that $2i+1$
appears no lower than row $n+i$ for $i\geq 0$.
As an example, when $\mu=(1)$ there is a single $1$-symplectic
Littlewood-Richardson tableau of weight $\beta=(1,1)$, namely
\[\young(\hfil 1,2)\]

We alert the reader that there is a typographical error  
in both~\cite[Definition~9.5]{MR2941115}
and~\cite[Definition~3.9]{MR1035496}, where the range of indices is
stated as $1\leq i\leq\frac{1}{2}\ell(\beta)$.  With this definition,
\[
  \young(\hfil,1,2)
\]
would also be $1$-symplectic, since $\beta=(1,1)$ and the only
relevant value for $i$ would be $1$, giving $2i+1=3$, which does not
appear at all in the tableau.  However, there are two $1$-symplectic
oscillating tableaux of length $3$ and shape $(1)$, and there are two
standard Young tableaux of shape $(2,1)$.  Thus, if the tableau above
were $1$-symplectic, Theorem~\ref{thm:Sundaram-bijection} below would
fail.

We now have all the definitions in place to explain the domain and
range of the correspondence $\Sun$.
\begin{theorem}[\protect{\cite[Theorem~9.4]{MR2941115}}]
  \label{thm:Sundaram-bijection}
  The map $\Sun$  described below is a bijection between $n$-symplectic
  oscillating tableaux of length $r$ and shape $\wgtCn$ and pairs
  $(Q, S)$, where
  \begin{itemize}
  \item $Q$ is a standard tableau of shape $\wgtSr$, with
    $|\wgtSr|=r$, and
  \item $S$ is an $n$-symplectic Littlewood-Richardson tableau of
    shape $\wgtSr/\wgtCn$ and weight $\beta$, where $\beta$ has even columns and
    $|\beta|=r-|\wgtCn|$.
  \end{itemize}
\end{theorem}

For completeness, let us point out the relation between Sundaram's
bijection and Berele's correspondence.  This correspondence involves
the following objects due to King~\cite{MR0480895}, indexing the
irreducible representations of the symplectic group $\Sp(2n)$.
\begin{definition}
  An \Dfn{$n$-symplectic semistandard tableau of shape $\wgtCn$} is a
  filling of $\wgtCn$ with letters from $1<-1<2<-2<\dots <n<-n$ such
  that
  \begin{itemize}
  \item entries in rows are weakly increasing,
  \item entries in columns are strictly increasing, and
  \item the entries in row $i$ are greater than or equal to $i$, in
    the above ordering.
  \end{itemize}
\end{definition}
Denoting the set of $n$-symplectic oscillating tableaux of length $r$
and final shape $\wgtCn$ by $\Osc(r, n, \wgtCn)$ and the set of
$n$-symplectic semistandard tableaux of shape $\wgtCn$ by
$\K(\wgtCn,n)$, Berele's correspondence is a bijection
\begin{equation*}
  \label{eq:Berele-explicit}
  \Ber: \{\pm 1,\dots,\pm n\}^r \to \bigcup_{\ell(\wgtCn)\leq n} \K(\wgtCn,n)\times\Osc(r, n, \wgtCn).
\end{equation*}
Analogous to  equation~\eqref{eq:Schur-character}, the character of
the representation $V(\wgtCn)$ of the symplectic group $\Sp(2n)$ is
\begin{equation}\label{eq:symplectic-character}
  sp_\wgtCn(x_1^{\pm1},\dots,x_n^{\pm1}) = \sum_{T\in\K(\wgtCn,n)} {\mathbf x}^T.
\end{equation}

Now consider an $n$-symplectic oscillating tableau as a word in the
ordered alphabet $1<-1<2<-2<\dots <n<-n$, as described just after
Definition~\ref{def:oscillating}.  We can then apply the
Robinson-Schensted correspondence to obtain a semistandard Young
tableau $P_{\RS}$ in this alphabet and an (ordinary) standard Young
tableau $Q_{\RS}$.  Alternatively, we can compose Berele's
correspondence with Sundaram's bijection to obtain a triple
$(P_{\Ber},Q_{\Sun},S_{\Sun})$.  It then turns out that
$Q_{\RS}=Q_{\Sun}$.  This implies that for each standard Young
tableau $Q_{\Sun}$ we have a correspondence $P_{\RS}\mapsto
(P_{\Ber}, S_{\Sun})$.  Moreover, this correspondence is independent
of the choice of $Q_{\Sun}$.  One can then prove the following
theorem.
\begin{theorem}[\protect{\cite[Theorem~12.1]{MR2941115}}]
  \label{thm:Sundaram-coefficients}
 If $V$ is the defining representation of $\Sp(2n)$ then the
  coefficient $a(\wgtSr,\wgtCn)$ is  given by
  \begin{equation*}
    a(\wgtSr,\wgtCn) = \sum_\beta c_{\wgtCn,\beta}^\wgtSr(n)
  \end{equation*} 
  where the sum is over the partitions $\beta$ of $|\wgtSr|-|\wgtCn|$
  having only columns of even length.
\end{theorem}

Let us point out a corollary, which settles a conjecture
from~\cite{Westbury2009}.
\begin{corollary}\label{cor:Frobenius-invariant}
 If $V$ is the defining representation of $\Sp(2n)$ then the  Frobenius
  character of the isotypic component of
$\otimes^rV$ of weight $0$ is given by
  \[
  \ch U(r, \emptyset)=\sum_{\beta} s_\beta,
  \]
  where the sum is over the partitions $\beta$ of $r$ such that if $c$ is any column length of $\beta$ then $c$ is even and $c\le 2n$.
\end{corollary}
\begin{proof}
Comparing equation~(\ref{eq:Frobenius}) with the previous Theorem,  it suffices to show that  $c_{\emptyset,\beta}^{\lambda}$ is nonzero only when $\lambda=\beta$ is a partition of the type described in the corollary, and in this case there is only one corresponding tableau.
  Suppose  $S$ is an $n$-symplectic Littlewood-Richardson tableau
  of shape $\wgtSr/\emptyset$ and weight $\beta$.  Then $S$ is a semistandard
  Young tableau of straight shape $\wgtSr$ whose reverse reading word
  is a lattice permutation.  This implies that all entries in row $j$
  of $S$ are equal to $j$.  In particular, $\wgtSr=\beta$ and $\beta$ satifies the conditions of 
Theorem~\ref {thm:Sundaram-coefficients}.
  Furthermore, is is required that the entries in row $j=n+i+1$ are
  at least $2i+2$ for $i\geq 0$.  It follows that $i+1\leq n$, and
  therefore $j\leq 2n$.
\end{proof}


\section{Proof of the main result}
\label{sec:proof}


In light of Sundaram's results, we claim that to prove Theorem~\ref{thm:main}
it suffices to demonstrate the following.

\begin{theorem}\label{thm:descents}
  Let $\Sun(\cT) = (Q, S)$.  Then
  \[
  \Des(\cT)=\Des(Q).
  \]
\end{theorem}

\begin{proof}[Proof of Theorem~\ref{thm:main}]
We have
\[
\begin{array}{rcll}
\displaystyle\sum_\cT \fQ_{\Des(\cT)}
&=&\displaystyle\sum_{(Q,S)} \fQ_{\Des(Q)}
&\hspace{-24pt}\text{(by Theorems~\ref{thm:Sundaram-bijection}
  and~\ref{thm:descents})}\\[20pt]
&=&\displaystyle\sum_{|Q|=r}\sum_{\beta} c_{\mu,\beta}^{\sh Q}(n)  \fQ_{\Des(Q)}
&\hspace{-24pt}\text{(by Definition~\ref{def:symplectic-LR})}\\[20pt]
&=&\displaystyle\sum_{|Q|=r} a(\sh(Q),\mu) \fQ_{\Des(Q)}
&\hspace{-24pt}\text{(by Theorem~\ref{thm:Sundaram-coefficients})}\\[20pt]
&=&\displaystyle\sum_{\wgtSr\in P(r)} a(\wgtSr,\mu)
\sum_{Q\in\SYT(\wgtSr)} \fQ_{\Des(Q)}
&\text{}\\[20pt]
&=&\ch U(r,\mu)
&\hspace{-24pt}\text{(by equations~\eqref{eq:quasisymmetric} and~\eqref{eq:Frobenius})}
\end{array}
\]
which is the desired conclusion.
\end{proof}

In order to prove Theorem~\ref{thm:descents}, we will need to analyse the bijection $\Sun$ in detail.  Sundaram described  $\Sun$ as the composition of
several bijections.  Specifically, $\Sun$ is the composition
\[
\cT\stackrel{\Sun_1}{\mapsto} (\iota,T)\stackrel{\RS}{\mapsto} (I,T)
\stackrel{\Sun_2}{\mapsto} (Q,S)
\]
where $\Sun_1$ and $\Sun_2$ are described below and, as usual, $\RS$ denotes the
Robinson-Schensted correspondence.  We will prove
Theorem~\ref{thm:descents} by tracking the effect  of these maps on the descent
set.

\subsection{Sundaram's first bijection}

We now describe Sundaram's first bijection which we will denote by
$\Sun_1$.  It maps an oscillating tableau $\cT$ to a pair $(\iota,T)$
where $\iota$ is a fixed-point-free involution, $T$ is a partial
Young tableau, and the entries of
$\iota$ and $T$ are complementary sets.

Let $\cT=(\emptyset\!=\!\wgtCn^0,\wgtCn^1,\dots,\wgtCn^r)$ be an
oscillating tableau.  We then construct a sequence of pairs
$(\iota_k,T_k)$ for $0\leq k\leq r$ such that $T_k$ has shape $\sh(T_k)=\wgtCn^k$,
and the entries of $\iota_k$ and of $T_k$ form a set partition of
$\{1,\dots,k\}$.

Both $\iota_0$ and $T_0$ are empty.  For $k>0$ the pair
$(\iota_k,T_k)$ is constructed from the pair $(\iota_{k-1},T_{k-1})$
and the $k$-th step in the oscillating tableau.
\begin{itemize}
\item If the $k$-th step is an expansion then $\iota_k=\iota_{k-1}$
  and $T_k$ is obtained from $T_{k-1}$ by putting $k$ in box $b_k$.  
\item If the $k$-th step is a contraction then take
  $T_{k-1}$ and bump out (using Robinson-Schensted column deletion) the entry in box
  $b_k$ to get a letter $x$ and the partial Young tableau $T_k$.  
   The involution $\iota_k$ is $\iota_{k-1}$ with the
 transposition $(x,k)$ adjoined.
\end{itemize}
The result of the bijection is the final pair $(\iota_r,T_r)$.

\begin{lemma}[Sundaram~\protect{\cite[Lemma~8.7]{MR2941115}}]
  The map $\Sun_1$ is a bijection between oscillating
  tableaux of length $r$ and shape $\wgtCn$ and pairs $(\iota, T)$
  where
  \begin{itemize}
  \item $\iota$ is a fixed-point-free involution of a set
    $A\subseteq\{1,\dots,r\}$, and
  \item $T$ is a partial tableau of shape $\wgtCn$ such that its set
    of entries is $\{1,\dots,r\}\setminus A$.
  \end{itemize}
\end{lemma}
In general it seems that, given the pair $(\iota, T)$, there is no
straightforward way to determine whether the corresponding
oscillating tableau is $n$-symplectic, see
\cite[Lemma~9.3]{MR2941115}.  However, when $\wgtCn=\emptyset$ the
map $\Sun_1$ is in fact a bijection between $n$-symplectic
oscillating tableaux and $(n+1)$-nonnesting perfect matchings of
$\{1,2,\dots,r\}$, see~\cite{ChenDengDuStanleyYan2006}.

\begin{example}\label{ex:sun1}
  Starting with the oscillating tableau $\cT$ from
  Example~\ref{ex:oscillating-tableaux}, we get the following sequence
  of pairs $(T_k,i_k)$, where in the diagram below we only list each
  pair of the involution once when it is produced by the algorithm.

  \[
  \begin{array}{lcccccccccc}
    \text{$k$:}  &0&1&2&3&4&5&6&7&8&9\\
    \totop{\text{$T_k$:}}%
    &\totop{\emptyset}
    &\totop{\young(1)}
    &\totop{\young(1,2)} 
    &\totop{\young(13,2)}
    &\totop{\young(13)}
    &\totop{\young(3)}
    &\totop{\young(36)}
    &\totop{\young(36,7)}
    &\totop{\young(36,7,8)}
    &\totop{\young(36,7)}\\
    \text{$\iota_k$:}%
    &\emptyset
    &&&&(2,4)&(1,5)&&&&(8,9)
  \end{array}
  \]

So the final output is the involution
\[
\iota =(1,5)(2,4)(8,9)=%
\begin{array}{cccccc}
1&2&4&5&8&9\\
5&4&2&1&9&8
\end{array}
\]
and the partial tableau
\[
\Yvcentermath1
T=\young(36,7)
\]
with $\iota\dotcup T=\{1,\dots,9\}$.
\end{example}

We wish to define the descent set of a pair $(\iota,T)$ in such a way
that this map preserves descents.  Let us define for any pair 
$A$ and $B$
\[
\Des(A/B)=\{k : \text{$k\in A$ and $k+1\in B$}\}.
\]
and
\[
\Des(\iota,T)=\Des(\iota)\dotcup \Des(T)\dotcup \Des(T/\iota),
\]
where we use Definition~\ref{sec:des-partial} for the descent sets of
$\iota$ and $T$.  Thus, in our running example,
\[
\Des(\iota)=\{1,4,8\},\
\Des(T)=\{6\},\
\Des(T/\iota)=\{3,7\},
\] 
so that $\Des(\iota, T)=\{1,3,4,6,7,8\}$, which
coincides with $\Des(\cT)$.

We can now take our first step in proving Theorem~\ref{thm:descents}.
\begin{proposition}\label{prop:Descents-Sun1}
  Let $\cT$ be an oscillating tableau and suppose that
  $\Sun_1(\cT)=(\iota, T)$.  Then
  \[ 
  \Des(\cT) = \Des(\iota,T).
  \]
\end{proposition}
\begin{proof}
  We proceed by analysing the effect of two successive steps in the
  oscillating tableau.

  If step $k$ is an expansion and step $k+1$ is a contraction then
  $k+1\in\iota$ and $\iota(k+1)< k+1$.  Now $k$ either ends up in $T$
  or in $\iota$.  In the former case, $k\in\Des(T/\iota)$.  In the
  latter case $\iota(k)> k\ge \iota(k+1)$ and $k\in\Des(\iota)$.  In both
  cases this gives a descent of $(\iota,T)$.

  If step $k$ is a contraction and step $k+1$ is an expansion then
  $k\in\iota$ and $\iota(k)<k$.  Now either $k+1$ ends up in $T$ or in
  $\iota$.  In the former case $k\in\Des(\iota/T)$ rather than $\Des(T/\iota)$.  In the latter case
  $\iota(k)< k<k+1<\iota(k+1)$.  Neither of these cases gives a
  descent of $(\iota,T)$.

  If steps $k$ and $k+1$ are both contractions then $k,k+1\in\iota$.
  If $b_k$ is strictly below $b_{k+1}$ then, by well-known properties
  of $\RS$, the element removed when bumping out $b_k$ will be in a
  lower row than the one obtained when bumping out $b_{k+1}$.  Thus
  $\iota(k)>\iota(k+1)$ and $k\in\Des(\iota,T)$ as desired.  By a similar
  argument, if $b_k$ is weakly above $b_{k+1}$ then $\iota(k)<\iota(k+1)$
  and $k\not\in\Des(\iota,T)$.

  Now suppose  steps $k$ and $k+1$ are both expansions.  If $b_k$ is in a row
  strictly above $b_{k+1}$, then any column deletion will keep $k$ in
  a row strictly above $k+1$.  It follows that at the end we have one
  of three possibilities.  The first is that $k,k+1\in T$ and, as was
  just observed, we must have $k\in\Des(T)$.  If either element is
  removed, then $k+1$ must be removed first because the row condition
  forces $k+1$ to always be in a column weakly left of $k$.  So at
  the end we either have $k\in T$ and $k+1\in\iota$, or we have
  $\iota(k)>\iota(k+1)$.  So in every case $k\in\Des(\iota,T)$.
By a similar argument, if $b_k$ is in a row
  weakly below $b_{k+1}$, then $k\not\in\Des(\iota,T)$.

This completes the check of all the cases and the proof.
\end{proof}

\subsection{Robinson-Schensted}

The next step of the map $\Sun$ is to apply the Robinson-Schensted correspondence to the
fixed-point-free involution $\iota$ to obtain a partial Young tableau
$I$ with the same set of entries as $\iota$.  Let us first recall the following
well-known facts about the Robinson-Schensted map.
\begin{lemma}[\protect{\cite[Exercise~22a]{SymmetricGroup} and~\cite[Exercise~7.28a]{EC2}}]\label{lem:RSK-properties}
  Let $\pi$ be a permutation with $\RS(\pi)=(P,Q)$.  Then
\begin{itemize}
\item $\Des(\pi)=\Des(Q)$, and
\item $\pi$ is a fixed-point-free involution if and only if $P=Q$ and
  all columns of $Q$ have even length.
\end{itemize}
\end{lemma}

Combining Proposition~\ref{prop:Descents-Sun1} with this lemma we
obtain the following result.
\begin{proposition}\label{prop:Descents-2}
  Let $\cT$ be an oscillating tableau, let $\Sun_1(\cT)=(\iota, T)$
  and let $\RS(\iota) = (I, I)$.  Then
  \[ 
  \Des(\cT) = \Des(I) \dotcup \Des(T) \dotcup \Des(T/I).
  \]
\end{proposition}
\begin{example}\label{ex:robinson-schensted}
  The fixed point free involution from Example~\ref{ex:sun1} is
  mapped to the tableau
  \[
  \Yvcentermath1
  I=\young(18,29,4,5)
  \]
\end{example}

\subsection{Sundaram's second bijection}

To finish defining $\Sun$,  we need a bijection $\Sun_2$ that transforms the pair
of partial  Young tableaux $(I, T)$ to a pair $(Q, S)$ as in
Theorem~\ref{thm:Sundaram-bijection}.

Let $Q$ be the standard Young tableau of shape $\wgtSr$ obtained by
\emph{column-inserting} the reverse reading word of the tableau $I$
into the tableau $T$.  Also construct a skew semistandard Young
tableau $S$ as follows: whenever a letter of $w^{rev}(I)$ is
inserted, record its row index in $I$ in the box which is added.
Then $\Sun_2(I, T) = (Q, S)$.
Sundaram actually defines this  map for pairs $(I, T)$ of semistandard
  Young tableaux.  But we will not need this level of generality.

\begin{lemma}[Sundaram~\protect{\cite[Theorem~8.11,
    Theorem~9.4]{MR2941115}}]
  The map $\Sun_2$ is a map from pairs of partial Young
  tableaux $(I, T)$ of shapes $\beta$ and $\wgtCn$, respectively,
  such that $I\dotcup T=\{1,\dots,r\}$ to pairs $(Q, S)$ such that
  \begin{itemize}
  \item $Q$ is a standard Young tableau of shape $\wgtSr\in P(r)$ and
  \item $S$ is a semistandard tableau of
    shape $\wgtSr/\wgtCn$ and weight $\beta$.
  \end{itemize}
  Furthermore, $\Sun_2$ becomes a bijection if one restricts the
  domain to those $(I,T)$ which correspond by $\Sun_1$ to an
  $n$-symplectic oscillating tableau, and the range to those $(Q,S)$
  where $S$ is an $n$-symplectic Littlewood-Richardson tableau.
\end{lemma}

\begin{example}
  The tableau $I$ from Example~\ref{ex:robinson-schensted} has
  $w^{\rev}(I)=819245$.  Column inserting this word into the tableau
  $T$ from Example~\ref{ex:sun1} and recording the row indices yields
  the pair of tableaux
  \[
  (Q,S)=%
  \left(\;%
    \Yvcentermath1%
    \young(136,27,48,59), \ %
    \young(\hfil\hfil1,\hfil2,13,24)%
    \;
  \right).
  \]
Note that the
 $\Des(Q)=\{1,3,4,6,7,8\}=\Des(\cT)$ as desired.
\end{example}

To prove Theorem~\ref{thm:descents} we need two properties of column
insertion.
\begin{lemma}\label{lem:RS-col-properties}
  Let $\pi$ be a permutation and suppose that $T$ is obtained by
  column inserting $\pi$.  Then 
  \[
  \Des(T) = \{k: \text{$k$ is left of $k+1$ in the $1$-line notation
    for $\pi$}\}.
  \]
  Furthermore, column inserting the reverse reading word
  $w^{\rev}(T)$ of $T$ yields $T$ itself.
\end{lemma}

\begin{proof}[Proof of Theorem~\ref{thm:descents}]
  By the first assertion of Lemma~\ref{lem:RS-col-properties} we have
  that $Q$ can also be obtained by column-inserting the concatenation
  of $w^{\rev}(T)$ and $w^{\rev}(I)$.  By the second assertion of
  Lemma~\ref{lem:RS-col-properties}, the descent set of $Q$ equals
  $\Des(I) \dotcup \Des(T) \dotcup \Des(T/I)$,
  which in turn equals the descent set of $\cT$ by
  Proposition~\ref{prop:Descents-2}.
\end{proof}


\section{Growth diagrams and Roby's description of $\Sun$}
\label{sec:Roby}

Using Fomin's framework of growth diagrams~\cite{Fomin1994,
  Fomin1995}, Roby~\cite{Roby1991} showed how to obtain the standard
tableau $Q$ directly from the oscillating tableau $\cT$.  However, he
omitted the proof that his construction does indeed correspond with
that of Sundaram's.  We take the opportunity here to review Roby's
construction and provide this proof.
Using his example Roby also attempted to obtain a direct description
of the $n$-symplectic Littlewood-Richardson tableau.  However, he
made the assumption that a permutation can be partitioned into
disjoint increasing subsequences whose lengths are given by the
lengths of the rows of the standard Young tableaux associated via the
Robinson-Schensted correspondence, which is not true in general.

\begin{figure}[h]
\[ \begin{tikzpicture}
\draw (0,0) node[anchor=north east]{$\lambda$} -- (0,1) node[anchor=south east]{$\mu$}
 -- (1,1) node[anchor=south west]{$\rho$} -- (1,0) node[anchor=north west]{$\nu$} -- cycle;
\draw (-.9,.5) node{$c=$};
\draw(2.2,.5) node{or};
\draw (4,0) node[anchor=north east]{$\lambda$} -- (4,1) node[anchor=south east]{$\mu$}
 -- (5,1) node[anchor=south west]{$\rho$} -- (5,0) node[anchor=north west]{$\nu$} -- cycle;
\draw (3.1,.5) node{$c=$};
\draw (4.5,.5) node{$X$};
\end{tikzpicture} \]
  \caption{A cell of a growth diagram.}
  \label{fig:growth-cell}
\end{figure}
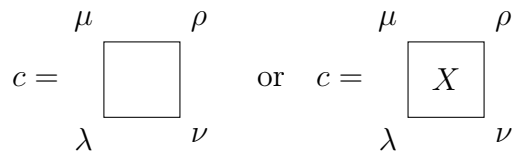

A \Dfn{growth diagram} is a rectangular grid sitting inside the first
quadrant of a Cartesian coordinate system and composed of cells, $c$
as in Figure~\ref{fig:growth-cell}.  The four corners of $c$ are
labeled with partitions as indicated, and if two partitions are
connected by an edge (e.g., $\lambda$ and $\nu$ in
Figure~\ref{fig:growth-cell}) then the one closer to the origin
either equals or is covered (in Young's lattice) by the one further
away.  In addition, each cell is either empty or contains a cross,
$X$, where we insist that there is at most one cross in every column
and every row of the diagram.  There are certain local rules
according to which the cells are labelled which we explain below.  An
example for a growth diagram is depicted in Figure~\ref{fig:ex}.

In the following we use the notation $\lambda\ltd\mu$ to denote that
$\lambda$ is covered by $\mu$.  Furthermore, for an integer $i$ and a
partition $\lambda$ we denote by $\lambda+\epsilon_i$ the partition
obtained from $\lambda$ by adding one to the $i$th part, assuming
that $\lambda_i<\lambda_{i-1}$.  Similarly we define $\lambda-\epsilon_i$.

The \Dfn{forward rules} determine the partition $\rho$ given the
other three partitions and whether the cell has a cross or not.  They
are as follows, organized by the relationship between $\mu$ and
$\nu$, and then secondarily by their relationship to $\lambda$.

\begin{itemize}
\item[F1] If $\mu\neq\nu$, then let $\rho=\mu\cup\nu$.
\item[F2] If $\lambda\ltd\mu=\nu$ then we must have $\mu=\lambda+\epsilon_i$ for some $i$, so let $\rho=\mu+\epsilon_{i+1}$.
\item[F3] If $\lambda=\mu=\nu$, then let
  \[
  \rho=
  \begin{cases}
    \lambda &\text{if $c$ does not contain an $X$,}\\
    \lambda+\epsilon_1 &\text{if $c$ contains an $X$.}
  \end{cases}
  \]
\end{itemize}

Thus, given the partitions $\lambda$, $\mu$ and $\nu$ and knowing
whether the cell contains a cross or is empty, we can compute the
partition $\rho$.  An important fact is that this process is
invertible in the following sense: given the partitions $\mu$, $\nu$
and $\rho$ we can compute the partition $\lambda$ as well as the
contents of the cell by using the following \Dfn{backward rules}.

\begin{itemize}
\item[B1] If $\mu\neq\nu$, then let $\lambda=\mu\cap\nu$.
\item[B2] If $\rho\gtd\mu=\nu$ then we must have $\mu=\rho-\epsilon_i$ for some $i$, so let
  \[
  \lambda=
  \begin{cases}
    \mu &\text{if $i=1$,}\\
    \mu-\epsilon_{i-1} &\text{if $i\ge2$.}    
  \end{cases}
  \]
  where $c$ gets filled with an $X$ in the first case and is left empty in the second.
\item[B3] If $\rho=\mu=\nu$ then let $\lambda=\mu$.
\end{itemize}

One can visualize these rules as follows.  For rules F1 and B1 we have the following possibilities, where in the first diagram $i\neq j$,
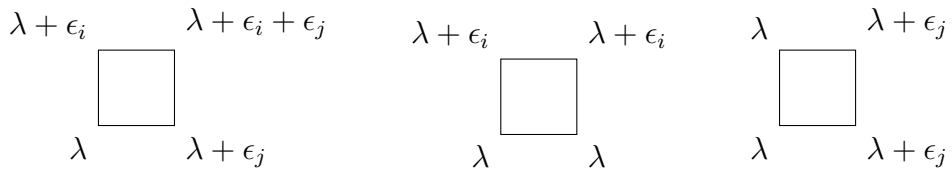
\[
 \begin{tikzpicture}
\draw (0,0) node[anchor=north east]{$\lambda$} -- (0,1) node[anchor=south east]{$\lambda+\epsilon_i$}%
 -- (1,1) node[anchor=south west]{$\lambda+\epsilon_i+\epsilon_j$}%
 -- (1,0) node[anchor=north west]{$\lambda+\epsilon_j$} -- cycle;
\end{tikzpicture} 
\qquad
\begin{tikzpicture}
\draw (0,0) node[anchor=north east]{$\lambda$} -- (0,1) node[anchor=south east]{$\lambda+\epsilon_i$}%
 -- (1,1) node[anchor=south west]{$\lambda+\epsilon_i$}%
 -- (1,0) node[anchor=north west]{$\lambda$} -- cycle;
\end{tikzpicture}  
\qquad
 \begin{tikzpicture}
\draw (0,0) node[anchor=north east]{$\lambda$} -- (0,1) node[anchor=south east]{$\lambda$}%
 -- (1,1) node[anchor=south west]{$\lambda+\epsilon_j$}%
 -- (1,0) node[anchor=north west]{$\lambda+\epsilon_j$} -- cycle;
\end{tikzpicture} 
\]
For rules F2--F3 and B2--B3, we can draw the following diagrams
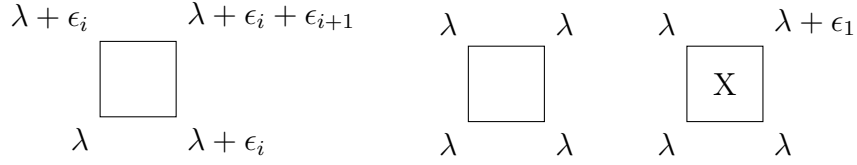
\[
 \begin{tikzpicture}
\draw (0,0) node[anchor=north east]{$\lambda$} -- (0,1) node[anchor=south east]{$\lambda+\epsilon_i$}%
 -- (1,1) node[anchor=south west]{$\lambda+\epsilon_i+\epsilon_{i+1}$}%
 -- (1,0) node[anchor=north west]{$\lambda+\epsilon_i$} -- cycle;
\end{tikzpicture} 
\qquad
 \begin{tikzpicture}
\draw (0,0) node[anchor=north east]{$\lambda$} -- (0,1) node[anchor=south east]{$\lambda$}%
 -- (1,1) node[anchor=south west]{$\lambda$} -- (1,0) node[anchor=north west]{$\lambda$} -- cycle;
\end{tikzpicture} 
\qquad
 \begin{tikzpicture}
\draw (0,0) node[anchor=north east]{$\lambda$} -- (0,1) node[anchor=south east]{$\lambda$}%
 -- (1,1) node[anchor=south west]{$\lambda+\epsilon_1$}%
 -- (1,0) node[anchor=north west]{$\lambda$} -- cycle;
\draw (0.5,0.5) node{X};
\end{tikzpicture} 
\]

The local rules are, in fact, just another way to perform
Robinson-Schensted insertion.  To explain this, consider  sequences of partitions 
$\ep=(\ep_0,\ep_1,\ldots,\ep_r)$ where $\ep_0=\emptyset$ and, for $k>0$, $\ep_k$ either equals $\ep_{k-1}$ or is obtained from $\ep_{k-1}$ by adding a box. Such a sequence $\ep$ corresponds to a partial Young tableau $P$ in the same way that an oscillating tableau with all steps expansions corresponds to a standard Young tableau: if  $\ep_k=\ep_{k-1}$ then $k$ does not appear in $P$; and if $\ep_k$ is obtained by adding a box to $\ep_{k-1}$ then $k$ appears in $P$ in the box added. Note that if the border of any row in a growth diagram starts with the empty partition, then the sequence of diagrams along this border satisfies the conditions for such an $\epsilon$.
\begin{theorem}[\protect{see \cite[Theorem~7.13.5]{EC2}}]
  \label{thm:RS-Fomin}
  Consider a growth diagram consisting of a single row whose cell in
  column $x$ contains a cross and whose left border is labeled with empty partitions.
 If the lower and upper borders are identified with partial tableaux $P$ and $P'$, respectively, then $P'$ is obtained by row inserting $x$ into $P$.
\end{theorem}
Note that, usually, this theorem is stated with the roles of
columns and rows interchanged.  However, this is without consequence
since the local rules are symmetric in this respect.

We can now explain Roby's description of Sundaram's correspondence.
As before, let $\cT=(\emptyset\!=\!\wgtCn_0,\wgtCn_1,\ldots,\wgtCn_r)$ be
an $n$-symplectic oscillating tableau.  We then proceed as follows.
\begin{enumerate}
\item[R1] Label the corners of the cells along the main diagonal,
  i.e., from north-west to south-east, of an $r\times r$ grid with
  the \emph{conjugates} of these partitions.
\item[R2] Using the rules B1--B4, determine the  partitions labelling the corners of the cells below the main diagonal and which of these cells contain a cross.  (Since neighboring partitions on
  the diagonal will always be distinct, rule B1 will always apply to
  determine the subdiagonal without needing to know $\rho$.)
\item[R3] Place crosses into those cells above the main diagonal,
  whose image under reflection about the main diagonal contains a
  cross.
\item[R4] Using the rules F1--F4, compute the partitions labelling
  the corners of the cells above the main diagonal.
\end{enumerate}

If we apply this procedure to our running example tableau,
\[
\cT=(\emptyset,1,11, 21, 2, 1, 2, 21,211,21),
\] 
then the corresponding oscillating tableau of conjugate (transposed)
partitions is
\[
\cT^t=(\emptyset,1,2,21,11,1,11,21,31,21)
\] 
and we obtain the growth diagram in Figure~\ref{fig:ex}.

\begin{figure}
\begin{tikzpicture}
\foreach \y in {1,...,10} 
   \draw (1,\y)--(10,\y);
\foreach \x in {1,2,...,10} 
   \draw (\x,1)--(\x,10);
\draw (1.4,5.4) node{$X$};
\draw (2.4,6.4) node{$X$};
\draw (8.4,1.4) node{$X$};
\draw (.7,9.7) node{$\emptyset$};
\draw (1.7,8.7) node{$1$};
\draw (2.7,7.7) node{$2$};
\draw (3.7,6.7) node{$21$};
\draw (4.7,5.7) node{$11$};
\draw (5.7,4.7) node{$1$};
\draw (6.7,3.7) node{$11$};
\draw (7.7,2.7) node{$21$};
\draw (8.7,1.7) node{$31$};
\draw (9.7,0.7) node{$21$};
\draw (.7,8.7) node{$\emptyset$};
\draw (1.7,7.7) node{$1$};
\draw (2.7,6.7) node{$2$};
\draw (3.7,5.7) node{$11$};
\draw (4.7,4.7) node{$1$};
\draw (5.7,3.7) node{$1$};
\draw (6.7,2.7) node{$11$};
\draw (7.7,1.7) node{$21$};
\draw (8.7,0.7) node{$21$};
\draw (.7,7.7) node{$\emptyset$};
\draw (1.7,6.7) node{$1$};
\draw (2.7,5.7) node{$1$};
\draw (3.7,4.7) node{$1$};
\draw (4.7,3.7) node{$1$};
\draw (5.7,2.7) node{$1$};
\draw (6.7,1.7) node{$11$};
\draw (7.7,0.7) node{$21$};
\draw (.7,6.7) node{$\emptyset$};
\draw (1.7,5.7) node{$1$};
\draw (2.7,4.7) node{$\emptyset$};
\draw (3.7,3.7) node{$1$};
\draw (4.7,2.7) node{$1$};
\draw (5.7,1.7) node{$1$};
\draw (6.7,0.7) node{$11$};
\draw (.7,5.7) node{$\emptyset$};
\draw (1.7,4.7) node{$\emptyset$};
\draw (2.7,3.7) node{$\emptyset$};
\draw (3.7,2.7) node{$1$};
\draw (4.7,1.7) node{$1$};
\draw (5.7,0.7) node{$1$};
\draw (.7,4.7) node{$\emptyset$};
\draw (1.7,3.7) node{$\emptyset$};
\draw (2.7,2.7) node{$\emptyset$};
\draw (3.7,1.7) node{$1$};
\draw (4.7,0.7) node{$1$};
\draw (.7,3.7) node{$\emptyset$};
\draw (1.7,2.7) node{$\emptyset$};
\draw (2.7,1.7) node{$\emptyset$};
\draw (3.7,0.7) node{$1$};
\draw (.7,2.7) node{$\emptyset$};
\draw (1.7,1.7) node{$\emptyset$};
\draw (2.7,0.7) node{$\emptyset$};
\draw (0.7,1.7) node{$\emptyset$};
\draw (1.7,0.7) node{$\emptyset$};
\draw (0.7,0.7) node{$\emptyset$};
\draw (4.4,8.4) node{$X$};
\draw (5.4,9.4) node{$X$};
\draw (9.4,2.4) node{$X$};
\draw (1.7,9.7) node{$1$};
\draw (2.7,8.7) node{$2$};
\draw (3.7,7.7) node{$21$};
\draw (4.7,6.7) node{$21$};
\draw (5.7,5.7) node{$11$};
\draw (6.7,4.7) node{$11$};
\draw (7.7,3.7) node{$21$};
\draw (8.7,2.7) node{$31$};
\draw (9.7,1.7) node{$31$};
\draw (2.7,9.7) node{$2$};
\draw (3.7,8.7) node{$21$};
\draw (4.7,7.7) node{$21$};
\draw (5.7,6.7) node{$21$};
\draw (6.7,5.7) node{$111$};
\draw (7.7,4.7) node{$21$};
\draw (8.7,3.7) node{$31$};
\draw (9.7,2.7) node{$41$};
\draw (3.7,9.7) node{$21$};
\draw (4.7,8.7) node{$31$};
\draw (5.7,7.7) node{$21$};
\draw (6.7,6.7) node{$211$};
\draw (7.7,5.7) node{$211$};
\draw (8.7,4.7) node{$31$};
\draw (9.7,3.7) node{$41$};
\draw (4.7,9.7) node{$31$};
\draw (5.7,8.7) node{$31$};
\draw (6.7,7.7) node{$211$};
\draw (7.7,6.7) node{$221$};
\draw (8.7,5.7) node{$311$};
\draw (9.7,4.7) node{$41$};
\draw (5.7,9.7) node{$41$};
\draw (6.7,8.7) node{$311$};
\draw (7.7,7.7) node{$221$};
\draw (8.7,6.7) node{$321$};
\draw (9.7,5.7) node{$411$};
\draw (6.7,9.7) node{$411$};
\draw (7.7,8.7) node{$321$};
\draw (8.7,7.7) node{$321$};
\draw (9.7,6.7) node{$421$};
\draw (7.7,9.7) node{$421$};
\draw (8.7,8.7) node{$331$};
\draw (9.7,7.7) node{$421$};
\draw (8.7,9.7) node{$431$};
\draw (9.7,8.7) node{$431$};
\draw (9.7,9.7) node{$441$};
\end{tikzpicture}
 \caption{The example growth diagram.}
  \label{fig:ex}
\end{figure}

Note that the partitions labelling the corners along the left border
must all be empty, since $\mu_0$, the first partition of the
oscillating tableau, is the empty partition.   So each row can be associated with a partial Young tableau.
  Let us single out two particular sequences of
partitions: $(\emptyset\!=\!\kappa_0,\kappa_1,\dots,\kappa_r)$,
labelling the corners along the upper border of the growth diagram,
and $(\emptyset\!=\!\tau_0,\tau_1,\dots,\tau_r)$ labelling the
corners along the lower border of the growth diagram.  In our running
example these are
\[
(\emptyset,1,2,21,31,41,411,421,431,441)
\]
and
\[
(\emptyset,\emptyset,\emptyset,1,1,1,11,21,21,21)
\]
respectively.

Now let
\begin{itemize}
\item $A_{\Rob}$ be the set of column indices of the growth diagram
  that contain a cross where indices are taken as in a matrix,
\item $\iota_{\Rob}$ be the involution of $A_{\Rob}$ that exchanges
  column and row indices of the cells containing crosses,
\item $T_{\Rob}$ be the partial Young tableau defined by
  $(\emptyset\!=\!\tau_0^t,\dots,\tau_r^t)$, and
\item $Q_{\Rob}$ be the partial Young tableau defined by
  $(\emptyset\!=\!\kappa_0^t,\dots,\kappa_r^t)$.
\end{itemize}
In our example, $A_{\Rob}=\{1,2,4,5,8,9\}$,
$\iota_{\Rob}=(1,5)(2,4)(8,9)$,
\[
\Yvcentermath1%
T_{\Rob}=\young(36,7)\quad\text{and}\quad Q_{\Rob}=\young(136,27,48,59).
\]

We now compute a second growth diagram.  Consider a square grid
having crosses in the same places as determined before.  Then, label
the corners along the lower and the left border of the grid with the
empty partition.  Finally, using the rules F1--F4 compute the
sequence of partitions $(\emptyset\!=\!\nu_0,\nu_1,\dots,\nu_r)$,
labelling the corners along the upper border of the growth diagram.
Let $I_{\Rob}$ be the partial Young tableau defined by
$(\emptyset\!=\!\nu_0^t,\dots,\nu_r^t)$.  

In our example, this sequence is $(\emptyset,
1,2,2,3,4,4,4,41,42)$ with associated tableau 
\[
\Yvcentermath1%
I_{\Rob}=\young(18,29,4,5).
\]

Comparing these objects with the ones computed using Sundaram's
correspondence, the reader will have anticipated the following
theorem.
\begin{theorem}\label{thm:Roby-Sundaram}
If
\[
\cT\stackrel{\Sun_1}{\mapsto} (\iota,T)\stackrel{RS}\mapsto (I,T) \stackrel{\Sun_2}{\mapsto} (Q,S)
\]
with $\iota$ an involution on $A$, then $A_{\Rob}=A$,
$\iota_{\Rob}=\iota$, $I_{\Rob}=I$, $T_{Rob}=T$ and $Q_{\Rob}=Q$
\end{theorem}
As mentioned at the beginning of this section it appears that there
is no straightforward way to extract the skew Littlewood-Richardson
tableau $S$ from the growth diagram.

For the proof the following simple observation will be helpful.
\begin{lemma}
  Suppose that the left border of a growth diagram is labelled by
  empty partitions and let $c$ be a cell labelled with partitions as
  in Figure~\ref{fig:growth-cell}.  Then $\lambda=\mu$ if and only if
  none of the cells to the left of $c$ and in the same row as $c$
  contains a cross.
\end{lemma}

\begin{proof}[Proof of Theorem~\ref{thm:Roby-Sundaram}]
  Let us first show that $A_{\Rob}=A$, $\iota_{\Rob}=\iota$ and
  $T_{\Rob}=T$.  We will use induction on the length $r$ of the
  oscillating tableau $\cT$.  There is nothing to show if $\cT$ is
  empty.  Thus, suppose that $r>0$ and that $\cT$ with the last step
  removed is mapped to the pair $(\iota^\prime, T^\prime)$ where
  $\iota'$ is an involution on $A'$.

  If the last step of $\cT=(\mu_0,\ldots,\mu_r)$ is an expansion,
  then we claim that the bottom row of the growth diagram does not
  contain a cross.  Since $\mu_{r-1}\ltd\mu_r$ the partition
  $\tau_{r-1}$ is obtained by applying rule B1 and we have
  $\tau_{r-1}=\mu_{r-1}$.  Now the claim follows from the lemma.

  We now have that $A_{\Rob}=A'$ and $\iota_{\Rob} = \iota' = \iota$
  in agreement with Sundaram's first correspondence $\Sun_1$.
  Moreover, since we have that  $\tau_{r-1}=\mu_{r-1}\ltd\mu_r=\tau_r$, $T_{\Rob}$ is obtained from
  $T'$ by putting $r$ into the cell added by the expansion, which
  coincides with $T$ as constructed via $\Sun_1$.

  If the last step of $\cT$ is a contraction, again using the lemma,
  the rule B1 entails that there must be a cross in the bottom row of
  the growth diagram, say in column $x$.  Thus, $\iota_{\Rob}$ is
  obtained from $\iota'$ by adjoining the pair $(x, r)$.  It remains
  to show that $T_{\Rob}$ is obtained by a column deletion of $T'$ where $x$ is the element removed at the end of the deletion process.  Thus we must show that column inserting $x$ into $T$
  yields $T'$, or equivalently, that row inserting $x$  into $T^t$ yields  $(T')^t$.  This last statement follows immediately from Theorem~\ref{thm:RS-Fomin}.

  We now turn to the computation of $Q$ and $Q_{\Rob}$.  Using
  $\Sun_2$ the tableau $Q$ is obtained by column inserting the
  reverse reading word of $I$ into $T$.  Equivalently, $Q^t$ is
  obtained by row inserting $w^{\rev}(I)$ into $T^t$.  Because
  $w^{rev}(I)$ is Knuth equivalent to the reversal of the involution
  $\iota$ in one line notation, Theorem~\ref{thm:RS-Fomin} shows that
  the tableau described by the partitions along the upper border of
  the growth diagram is indeed $Q^t$.

  The equality of $I$ and $I_{\Rob}$ also follows directly from
  Theorem~\ref{thm:RS-Fomin}.
\end{proof}

One nice feature of the growth diagram formulation of Sundaram's
correspondence is that the descent set can be visualized.  Namely,
take \emph{any} partial permutation with insertion tableau $T^t$
and construct its growth diagram.  On top of this stack the growth
diagram for the oscillating tableau so that the two rows corresponding to $T^t$ coincide.  Then $k$ is a descent in the
oscillating tableau if and only if the cross in column $k$ is lower than the cross in column $k+1$.
 It is in fact possible to prove Theorem~\ref{thm:descents}
using this approach.

\providecommand{\cocoa} {\mbox{\rm C\kern-.13em o\kern-.07em C\kern-.13em
  o\kern-.15em A}}\def\cprime{$'$}
\providecommand{\bysame}{\leavevmode\hbox to3em{\hrulefill}\thinspace}
\providecommand{\MR}{\relax\ifhmode\unskip\space\fi MR }
\providecommand{\MRhref}[2]{%
  \href{http://www.ams.org/mathscinet-getitem?mr=#1}{#2}
}
\providecommand{\href}[2]{#2}

\end{document}